\documentclass[12pt]{amsart}
\numberwithin{equation}{section}
\newtheorem{theorem}{Theorem}[section]
\newtheorem{lemma}[theorem]{Lemma}
\newtheorem{corollary}[theorem]{Corollary}
\newtheorem{proposition}[theorem]{Proposition}
\newtheorem{conjecture}[theorem]{Conjecture}

\theoremstyle{definition}
\newtheorem{definition}[theorem]{Definition}
\newtheorem{example}[theorem]{Example}
\newtheorem{problem}[theorem]{Problem}

\theoremstyle{remark}

\newcommand{\N}{\mathbb{N}}
\DeclareMathOperator{\gap}{gap}

\usepackage{amssymb}
\usepackage{caption}
\usepackage{fullpage}
\usepackage{amsmath}
\usepackage{amsfonts}
\usepackage{latexsym}
\usepackage{hyperref}

\title{Novel structures in Stanley sequences}
\author{Richard A. Moy}
\thanks{Northwestern University, Evanston IL. Email: \texttt{ramoy88@math.northwestern.edu}}
\author{David Rolnick}
\thanks{Massachusetts Institute of Technology, Cambridge MA. Email: \texttt{drolnick@math.mit.edu}}

\begin{document}

\begin{abstract}
Given a set of integers with no three in arithmetic progression, we construct a \emph{Stanley sequence} by adding integers greedily so that no arithmetic progression is formed. This paper offers two main contributions to the theory of Stanley sequences. First, we characterize well-structured Stanley sequences as solutions to constraints in modular arithmetic, defining the \emph{modular Stanley sequences}.  Second, we introduce the \emph{basic Stanley sequences}, where elements arise as the sums of subsets of a \emph{basis} sequence, which in the simplest case is the powers of 3. Applications of our results include the construction of Stanley sequences with arbitrarily large gaps between terms, answering a weak version of a problem by Erd\H{o}s et al. Finally, we generalize many results about Stanley sequences to $p$-free sequences, where $p$ is any odd prime.
\end{abstract}

\maketitle


\section{Introduction}

A set of nonnegative integers is said to be \emph{$p$-free} if it contains no $p$-term arithmetic progressions. There is great interest in finding the maximum cardinality $r_p(n)$ of a $p$-free subset of $\{0,1,\ldots,n\}$. Some work has gone into computing explicit values of $r_p(n)$ for small values of $p$ and $n$ \cite{ET,Wag1,Wag2,W,GGK,D,ADS}. However, the vast majority of research about $r_p(n)$ has involved its asymptotic behavior. In 1936, Erd\H{o}s and Tur\'an \cite{ET} stated a conjecture of Szekeres that, for an odd prime $p$, implied $r_p(n)=\Theta(n^{\log_p(p-1)})$. Szekeres' conjecture was disproven by Salem and Spencer \cite{SS} in 1942 when they showed that as $n\rightarrow\infty$, one has $r_3(n)>n^{1-\frac{\log 2+\epsilon}{\log\log n}}$ for every $\epsilon>0$.

This bound was improved by Behrend \cite{B} in 1946, when he proved that as $N\rightarrow\infty$, one has $r_3(n)>n^{1-\frac{2\sqrt{2\log 2}+\epsilon}{\sqrt{\log n}}}$ for every $\epsilon>0$. This result was slightly improved by Elkin \cite{E} in 2008, and a shorter proof of Elkin's result was soon discovered by Green and Wolf \cite{GW}. Unfortunately, these lower bounds for $r_3(n)$ are far from the upper bounds $r_3(n)=O\left(\frac{n}{\log \log n}\right)$ achieved by Roth \cite{Roth}, $r_3(n)=O\left(\frac{n}{(\log n)^{\frac{1}{20}}}\right)$ achieved by Heath-Brown \cite{HB}, and then $r_3(n)=O\left(\frac{n (\log\log n)^5}{\log n} \right) $ achieved by Sanders \cite{S}.

In 1978, Odlyzko and Stanley \cite{OS} proposed constructing $3$-free sequences according to the greedy algorithm.

\begin{definition}
Let $A=\{a_0,\ldots,a_k\}$ be a $3$-free set of nonnegative integers satisfying $a_0<\cdots<a_k$. We define the \emph{Stanley sequence} $S(A)=(a_n)$ generated by $A$ recursively as follows. If $a_0<\cdots<a_n$ have already been defined, then $a_{n+1}$ is the smallest positive integer greater than $a_n$ such that $\{a_0,\ldots,a_n,a_{n+1}\}$ is $3$-free.

In a slight abuse of notation, we will often write $S(a_0,\ldots,a_k)$ for $S(\{a_0,\ldots,a_k\})$.  We will also sometimes consider the sequence $S(A)$ as a set.
\end{definition}

The name ``Stanley sequences'' originates with Erd\H{o}s et al.~\cite{ELRSS}, who generalized the procedure above from the case of $|A|=2$, as originally proposed by Odlyzko and Stanley. The simplest Stanley sequence is $S(0)=0,1,3,4,9,10,12,13,27,\ldots,$ the elements of which are exactly those integers with no 2's in their ternary expansion; the growth rate of this sequence matches the proposed bound of Szekeres and indeed motivated this conjecture. Remarkably, Stanley sequences appear to exhibit two distinct patterns of asymptotic growth \cite{OS,EG,ELRSS}, with no intermediate growth rate possible.  While some sets $A$, such as $A=\{0\}$, lead to highly ordered Stanley sequences $S(A)$ (\emph{Type 1} growth), others lead to chaotic sequences (\emph{Type 2} growth). The original conjecture of Odlyzko and Stanley was generalized by Rolnick in \cite[Conjecture 1.1]{R} as follows.

\begin{conjecture}
\label{conj:main}
Let $S(A)=(a_n)$ be a Stanley sequence. Then, for all $n$ large enough, one of the following two patterns of growth is satisfied.

\begin{enumerate}
\item $\alpha/2\le \liminf a_n/n^{\log_2 3}\le \limsup a_n/n^{\log_2 3}\le \alpha$, or
\item{ $a_n=\Theta(n^2\slash\log n).$ }
\end{enumerate}

\end{conjecture}

Although only the case $\alpha=1$ was considered by Odlyzko and Stanley \cite{OS}, Rolnick and Venkataramana have shown \cite{RV} that every rational number $\alpha=1$ is possible for which the denominator is a power of 3. Odlyzko and Stanley showed that \emph{Type 2} growth is, in some sense, the ``expected'' growth of a Stanley sequence, assuming that elements occur in the sequence according to a continuous probability distribution. This justifies the formula $\Theta(n^2\slash\log n)$, which has been experimentally verified by Lindhurst \cite{L} up to large values of the sequence $S(0,4)$. However, no Stanley sequence, including $S(0,4)$, has been definitively proven to satisfy \emph{Type 2} growth. In this paper, we will study the behavior of \emph{Type 1} sequences, thus making progress towards Conjecture \ref{conj:main}.

Erd\H{o}s et al.~posed several problems on the the asymptotic behavior of Stanley sequences. In \cite{M}, Moy solved Problem 1 of \cite{ELRSS} by proving that in any Stanley sequence $(a_n)$, the terms $a_n$ grow no faster than $n^2/(2+\epsilon)$, where $\epsilon$ is an arbitrary constant. An effective lower bound (Problem 2) remains open; that is, proving $\liminf \log a_n/\log n<1$, where the $\liminf$ is conjectured to be $\log_2 3$.

Erd\H{o}s et al.~also consider the gaps between consecutive elements, asking whether there exists a Stanley sequence $(a_n)$ for which $\liminf (a_{n+1}-a_n)=\infty$ (Problem 4, \cite{ELRSS}). A weaker version of this question (Problem 6) was answered in the affirmative by Savchev and Chen \cite{SC}, who constructed a 3-free sequence $(a_n)$ satisfying $\liminf (a_{n+1}-a_n)=\infty$ for which no integer can be added without violating the 3-free property; this sequence is however not a Stanley sequence. Among the results in this paper, we show that there exist Stanley sequences for which $\liminf (a_{n+1}-a_n)$ is arbitrarily large.

As noted in Erd\H{o}s and Graham \cite[page 22]{EG}, sequences like $S(0,4)$ seem to admit no closed-form description. However, Rolnick \cite{R} has extensively studied sequences that exhibit \emph{Type 1} growth, constructing many novel sequences of this form.  In particular, \cite{R} introduces the concept of independent and regular Stanley sequences, which satisfy \emph{Type 1} growth. Rolnick conjectures that these are in fact the \emph{only} \emph{Type 1} sequences; in this paper, however, we will show that the definitions must be slightly modified for this conjecture to hold.

\begin{definition}\label{def:indep}
A Stanley sequence $S(A)=(a_n)$ is \emph{independent} if there exist constants $\lambda=\lambda(A)$ and $\kappa=\kappa(A)$ such that for all $k\geq \kappa$ and $0\leq i<2^k$, we have

\begin{itemize}
\item $a_{2^k+i}=a_{2^k}+a_i$,
\item $a_{2^k}=2a_{2^k-1}-\lambda+1$.
\end{itemize}

The constant $\lambda$ is referred to as the \emph{character}; it is proven in \cite{R} that $\lambda\ge 0$ for all independent Stanley sequences. If $\kappa$ is taken as small as possible, then $a_{2^\kappa}$ is called the \emph{repeat factor}; informally, it is the point at which the sequence begins its repetitive behavior. It is proven in \cite{RV} for any sufficiently large integer $\rho$, there is an independent Stanley sequence with repeat factor $\rho$.
\end{definition}

\begin{example}
The sequence $S(0,1,7)$ is independent, with character $\lambda=7$ and repeat factor $a_4=10$.
$$
S(0,1,7)=0,1,7,8,10,11,17,18,30,31,37,38,40,41,47,48,90,\ldots
$$
Notice that the terms $a_4,a_5,a_6,a_7$ are simply the terms $a_0,a_1,a_2,a_3$ increased by 10. Likewise, terms $a_8$ through $a_{15}$ equal the terms $a_0$ through $a_7$ increased by 30; this illustrates the first condition of an independent sequence.  Furthermore, the sequence approximately doubles when the index is a power of 2: between $a_3=8$ and $a_4=10$, between $a_7=18$ and $a_8=30$, and between $a_{15}=48$ and $a_{16}=90$. These jumps become increasingly evident as the index increases, since the character $\lambda$ represents a correction; for instance, $a_{16}=2\cdot a_{15}-\lambda+1$.
\end{example}

\begin{definition}
A Stanley sequence $S(A)=(a_n)$ is \emph{regular} if there exist constants $\lambda$, $\sigma$ and an independent Stanley sequence $(a_n')$, having character $\lambda$, such that, for large enough $k$ and $0\le i<2^k$,
\begin{itemize}
\item $a_{2^k-\sigma+i}=a_{2^k-\sigma}+a_i$,
\item $a_{2^k-\sigma}=2a_{2^k-\sigma-1}-\lambda+1$.
\end{itemize}

The sequence $(a_n')$ is called the \emph{core} of $S(A)$ and the constant $\sigma$ is the \emph{shift index}. We refer to $\lambda$ as the \emph{character} of $(a_n)$ as well as of $(a_n')$.
\end{definition}

\begin{example}
The sequence $S(0,1,4)$ is regular with core $(a'_n)=S(0)$, shift index $\sigma=0$, and character $\lambda=0$.
$$
S(0,1,4)=0,1,4,5,11,12,14,15,31,32,34,35,40,41,43,44,89,\ldots
$$
Notice that the terms $a_4,a_5,a_6,a_7$ equal the terms $a'_0,a'_1,a'_2,a'_3$ of $S(0)$, translated by 11.  Likewise, the terms $a_8$ through $a_{15}$ equal the terms $a'_0$ through $a'_7$ translated by 31. As with an independent sequence, the sequence $S(0,1,4)$ has jumps when the index is a power of 2, for instance from 44 to 89. However, for some regular sequences, the shift index is nonzero, which means that the jumps are shifted away from powers of 2.  For instance, by removing 11 from the sequence $S(0,1,4)$, while leaving all other terms unchanged, we obtain the sequence $S(0,1,4,5,12,14,15,31)$, which is also regular and has shift index 1. In a sense, the term 11 is ``unnecessary'' in the Stanley sequence.
\end{example}

In this paper, we connect the asymptotic behavior of \emph{Type 1} Stanley sequences to their properties under modular arithmetic. We show that every independent sequence is not simply 3-free, but is also 3-free modulo a certain integer. This idea allows us to generalize independent sequences to \emph{modular sequences}; our definition of \emph{pseudomodular sequences} generalizes regular sequences in the same way. We describe modular Stanley sequences with novel structure (Table \ref{table:notindep}), and in particular identify a beautiful class of Stanley sequences created by summing the subsets of other sequences. We call these sequences \emph{basic sequences}, since in some sense they generalize the notion of writing an integer in base 3.

Our results are organized as follows. In \S \ref{sec:mod}, we present modular Stanley sequences and describe their structure (Theorem \ref{thm:struc}) and asymptotic behavior. In \S \ref{sec:pseudomod}, we present pseudomodular sequences, which generalize the concept of regular sequences in the same way that modular sequences generalize the concept of independent sequences. Theorem \ref{thm:pseudomod} shows how to construct complicated pseudomodular Stanley sequences by translating parts of a modular sequence.

In \S \ref{sec:basic}, we present basic sequences and show how to construct them (Theorem \ref{thm:basic}), using our results for modular sequences.  We also show how to construct modular sequences from basic sequences (Theorem \ref{thm:scalebasic}). In \S \ref{sec:gaps}, we apply our previous results to the problem of constructing Stanley sequences with arbitrarily large gaps between consecutive terms (Corollary \ref{cor:liminf}), partially addressing a question of Erd\H{o}s et al \cite{ELRSS}.  In \S \ref{sec:pfree}, we show how our results generalize to $p$-free integer sequences. Finally, in \S \ref{sec:future}, we discuss possible future directions of research in this area, and pose several conjectures.

\section{Modular sequences}
\label{sec:mod}
In Rolnick's definition \cite{R} of independent Stanley sequences stated above in Definition \ref{def:indep}, the Stanley sequence is constructed by repeatedly translating a set of $2^k$ integers (the integers less than the repeat factor). As we shall see, the condition that this set contain $2^k$ integers is not essential and is merely due to the fact that this is the case for ``simple'' Stanley sequences satisfying \emph{Type 1 growth}. Essentially, the definition of a modular sequence removes this condition. In Table \ref{table:notindep}, we will show examples of modular sequences which are not independent (or regular). 

\begin{definition}
Let $A$ be a set of integers and $x$ be an integer. We say that $x$ is \emph{covered} by $A$ if there exist $y,z\in A$ such that $z<y$ and $2y-z=x$.

Suppose that $N$ is a positive integer with $A\subseteq \{0,\ldots,N-1\}$. Then, we say that $x$ is \emph{covered} by $A$ \emph{modulo $N$} if there exist $y,z\in A$ with $z<y$ such that $2y-z\equiv x\pmod N$.
\end{definition}

\begin{definition}
Fix a positive integer $N\ge 1$. Suppose there exists a set $A\subset\{0,\ldots,N-1\}$ containing $0$ such that $A$ is 3-free modulo $N$ and all $x\in\{0,\ldots,N-1\}\backslash A$ are covered by $A$ modulo $N$. Then, $A$ is said to be a \emph{modular set modulo $N$} and $S(A)$ is said to be a \emph{modular Stanley sequence modulo $N$}.
\end{definition}

\begin{proposition}
Suppose $A$ is a finite subset of $\N_0$ and suppose $S(A)$ is an independent Stanley sequence with repeat factor $\rho$. Then $S(A)$ is a modular Stanley sequence modulo $3^\ell\cdot \rho$ for some integer $\ell\ge 0$.
\end{proposition}

\begin{proof}
Let $\omega(A)$ denote the largest integer that is neither in $S(A)$ nor is covered by $S(A)$. Choose $\ell$ such that $N:=3^\ell\cdot \rho$ satisfies $N-\max\left(S(A)\cap \{0,1,\ldots,N-1\}\right)>\omega(A)$ and $\max\left(S(A)\cap\{0,1,\ldots,N-1\}\right)>\omega(A)$. Let $A'=S(A)\cap\{0,1,\ldots,N-1\}$. We know that $S(A')=S(A)$. This is easy to see since $\max(A')>\omega(A)$ and $A'=S(A)\cap\{0,1,\ldots,N-1\}$.

We claim that every element $x\in\{0,1,\ldots,N-1\}\backslash A'$ is covered by $A'$ modulo $N$. Clearly if $x$ is covered by $A'$, then $x$ is covered by $A'$ modulo $N$. Note that $x>\omega(A)$ implies $x$ is covered by $A'$. Now suppose that $x\le\omega(A)$ and $x$ is not covered by $A'$. Then, $x+N\not\in S(A')$ and $x+N$ must be covered by $S(A')$. Hence, we have a 3-term arithmetic progression (3-AP) of the form $z<y<x+N$, with $z,y\in S(A')$.

First, suppose $z\ge N$. Because $S(A')$ is independent with repeat factor $\rho$ and $N=3^\ell\cdot \rho$, we must have $z=z'+N$ and $y=y'+N$ where $z',y'\in A'$. Then $x$ is covered by $z'<y'<x$, a contradiction. Now, suppose $z<N$ and $y\ge N$. Then
$$
x+N=2y-x\ge 2\cdot N-\max(A')>N+\omega(A)\ge x+N,
$$
again a contradiction. We conclude that $y,z<N$; therefore $y,z\in A'$ and $z<y$ cover $x$ modulo $N$.

It remains to show that $A'$ is 3-free modulo $N$. Suppose we have $x,y,z\in A'$ such that $2y-z\equiv x\pmod N$. Observe that $-N<2y-z<2N$. If $0\le 2y-z<N$, then $z,y,x$ would form a 3-AP in $A'$, a contradiction. If $2y-z<0$, then we have $x=2y-z+N$ and therefore $z,y+N,x+N$ form a 3-AP in $S(A')$, a contradiction. If $2y-z\ge N$, then $x=2y-z-N$, and therefore $z,y,x+N$ form a 3-AP in $S(A')$, again a contradiction.

We conclude that $S(A)$ is a modular Stanley sequence modulo $N=3^\ell\cdot \rho$.
\end{proof}

We now describe the structure of modular Stanley sequences, which generalizes a similar result of Rolnick and Venkataramana \cite[Proposition 2.2]{RV} for independent sequences. For sets $X,Y$ and constant $c$, we will use the notation $X+Y$ to refer to the set $\{x+y\mid x\in X,y\in Y\}$ and $c\cdot X$ to refer to $\{cx\mid x\in X\}$.

\begin{theorem}
\label{thm:struc} Suppose that $A$ is a modular set modulo $N$, with $N$ a positive integer.

(i) $S(A)=A+N\cdot S(0)$.

(ii) Pick $\alpha\in\N$ such that $\gcd(\alpha,N)=1$. Then 
$$
\alpha\cdot A + N\cdot S(0)
$$
is a modular Stanley sequence modulo $3^{\ell}\cdot N$ for some integer $\ell\ge 0$.
\end{theorem}

\begin{example}
Take $A=\{0,1,7,8 \}$, which is a modular set modulo $10$. We have $S(A)=A+10\cdot S(0)$. Picking $\alpha=9$, we may verify that $9\cdot A+10\cdot S(0)$ is a modular Stanley sequence and is in fact equal to 
$$
S(0,9,10,19,30,39,40,49,63,72,73,82,90,93,99).
$$
\end{example}

This theorem immediately implies that every modular sequence $S(A)$ follows Type 1 growth, because $S(0)$ follows Type 1 growth. Also observe that the theorem implies that if $S(A)$ is a modular Stanley sequence modulo $N$, then it is also a modular Stanley sequence modulo $3N$.

\begin{corollary}
\label{cor:modgrowth}
Every modular sequence follows Type 1 growth.
\end{corollary}

\begin{proof}[Proof of Theorem \ref{thm:struc}]
We first show how part (i) follows from part (ii).  For $\ell\in \N$, let $S_\ell$ denote the first $2^\ell$ elements of $S(0)$, with maximum element $(3^\ell-1)/2$, and let $A_\ell:=\alpha\cdot A+N\cdot S_\ell$. Pick $\ell$ large enough that $\max(A_\ell)\le 3^\ell\cdot N$. We will prove part (ii) for this choice of $\ell$. Taking $\alpha=1$ and $\ell=0$ yields part (i).

We now prove part (ii); that is, $A_{\ell}$ is a modular set modulo $3^\ell \cdot N$. Let $S=\alpha\cdot A + N\cdot S(0)$. First, we prove that $S$ is 3-free. Suppose towards contradiction that there exists a 3-AP $z,y,x$ in $S$. Write $x=\alpha x_A+Nx_0$, for $x_A\in A$, $x_0\in S(0)$, and define $y_A,z_A,y_0,z_0$ similarly. Then, $z_A,y_A,x_A$ form a 3-AP modulo $N$, because $\gcd(\alpha,N)=1$. Since $A$ is 3-free modulo $N$, we conclude that $x_A=y_A=z_A$. Therefore, $z_0,y_0,x_0$ must form a 3-AP, a contradiction since $S(0)$ is 3-free.

Now pick $x\not\in S$ such that $x\ge 3^\ell\cdot N$. We must show that $x$ is covered by $S$. Let $x_A$ be the unique element of $\{0,\ldots,N-1\}$ such that $x\equiv \alpha x_A\pmod{N}$. (Note that $x_A$ only exists because $\gcd(\alpha,N)=1$.)  Define $y_A,z_A,y_0,z_0$ as follows.

If $x_A\in A$, then set $y_A=z_A=x_A$. Else, since $A$ is a modular set modulo $N$, there exist $z_A,y_A\in A$ such that $z_A,y_A,x_A$ form a 3-AP. Now, define $x_0$ by $x=\alpha(2y_A-z_A)+Nx_0$. Note that the integer $x_0$ is nonnegative, because $x>\max(\alpha\cdot A)$.

Now, if $x_0\in S(0)$, set $y_0=z_0=x_0$. Else, since $S(0)$ is a Stanley sequence, there exist $z_0,y_0\in A$ such that $z_0,y_0,x_0$ form a 3-AP. Setting $y=\alpha y_A+Ny_0$ and $z=\alpha z_A+Nz_0$, we see that $y,z\in S$ and $z,y,x$ form a 3-AP, completing our proof.
\end{proof}

In \cite[Theorem 1.3]{R}, Rolnick introduces an operator $\otimes$, called the \emph{product}, that combines regular Stanley sequences to produce another regular Stanley sequence. We translate this definition of $\otimes$ to modular sequences here. The proof is similar to the proofs of the preceding theorem.

\begin{proposition}
Suppose that $A$ and $B$ are modular sets, modulo $M$ and $N$ respectively. Let $A\otimes B=A+M\cdot B$. Then $S(A\otimes B)$ is a modular Stanley sequence modulo $MN$.
\end{proposition}

\begin{example}
Set $A=\{0,1,7,8\}$, which is a modular set modulo $10$, and set $B=\{0,2\}$, which is a modular set modulo $3$. The set
$$
A\otimes B=\left\{0,1,7,8,20,21,27,28 \right\}
$$
is a modular set modulo $30$.
\end{example}

Note that $\otimes$ is an associative operation on modular sequences.  Namely, for $A,B,C$ modular sets, modulo $M,N,P$, the set $(A\otimes B)\otimes C=A\otimes (B\otimes C)$ is modular with modulus $MNP$. Thus, the set of modular sets is a noncommutative monoid under the operation $\otimes$ and has identity $\{0\}$.


\section{Pseudomodular sequences}
\label{sec:pseudomod}

\begin{table}[!ht]
$$\begin{array}{|c|c|c|}
\hline
A & N & |A| \\
\hline\hline
0, 6, 13, 14, 16, 17, 27, 29, 30, 35, 36, 49, 50 & 61 & 13 \\
\hline
0, 1, 3, 4, 9, 12, 26, 29, 34, 37, 50, 53, 60, 61  & 75 & 14\\
\hline
 0, 5, 7, 8, 12, 13, 15, 20, 29, 36, 44, 55, 62, 63, 70  & 79 & 15\\
\hline
 0, 7, 12, 16, 18, 23, 35, 41, 42, 53, 57, 62, 69, 74, 78, 80, 92 & 93 & 17\\
\hline
0, 11, 12, 16, 18, 19, 40, 46, 48, 57, 59, 65, 86, 87, 89, 93, 94, 105, 110, 120, 122 & 125 & 21\\
\hline
0, 11, 12, 23, 45, 56, 57, 68, 84, 95, 110, 116, 121, 127, 142, 153, 194,  & 673 & 63 \\
207,213,218,224,226, 237, 265, 271, 276, 282, 297, 303, 308, 314, 355,  & & \\
356, 366, 374, 375,385, 386, 427, 433, 438,444, 459, 465, 470, 476, 504, & & \\
515,517,523,528,534, 536, 547,588, 599, 614, 620, 625, 631, 646, 657& &\\
\hline
\end{array}$$
\caption{Modular sequences $S(A)$ that are not independent, with their modulus $N$ shown and the cardinality of the modular set $A$. Several of these examples were found by modifying sequences on a website of Wroblewski \cite{W}.}
\label{table:notindep}
\end{table}

Rolnick conjectured that all Stanley sequences following Type 1 growth are regular. The modular sequences shown in Table \ref{table:notindep} disprove this conjecture, since they are not independent (or regular). However, only a slight modification is necessary in Rolnick's original definition of independent and regular sequences to resurrect the conjecture: The period of such a sequence need not in fact be a power of 2. We restate the definition of a modular sequence in these terms.

\begin{definition}
A Stanley sequence $S(A)=(a_n)$ is \emph{modular} if there exist constants $\lambda,m$ such that for all $k\geq 0$ and $0\leq i<m\cdot 2^k$,
\begin{itemize}
\item $a_{m\cdot 2^k+i}=a_{m\cdot 2^k}+a_i$,
\item $a_{m\cdot 2^k}=2a_{m\cdot 2^k-1}-\lambda+1$.
\end{itemize}

We refer to $\lambda$ as the \emph{character} and $a_m$ as the \emph{repeat factor}, where $m$ is assumed to be as small as possible so that the above conditions are satisfied.
\end{definition}

Just as the definition of a modular sequence is a slight modification of the definition of an independent sequence, so we can modify the definition of a regular sequence to reflect that the period need not be a power of 2.

\begin{definition}
A Stanley sequence $S(A)=(a_n)$ is \emph{pseudomodular} if there exist constants $\lambda, m$ and a modular Stanley sequence $(a_n')$, having character $\lambda$, such that, for all $k\ge 0$ and $0\le i<2^k$,

\begin{itemize}
\item $a_{m\cdot 2^k-\sigma+i}=a_{m\cdot 2^k-\sigma}+a_i$,
\item $a_{m\cdot 2^k-\sigma}=2a_{m\cdot 2^k-\sigma-1}-\lambda+1$.
\end{itemize}

We refer to the sequence $(a_n')$ as the \emph{core} of $S(A)$ and the constant $\sigma$ as the \emph{shift index}.
\end{definition}

It is easy to see that every regular Stanley sequence (according to Rolnick's definition) is pseudomodular, with $m$ taken to be a power of 2. The statement and proof of Theorem 1.5 in \cite{R} generalize naturally from regular sequences to pseudomodular sequences.

\begin{theorem}
\label{thm:pseudomod}
Let $S(A)=(a_n)$ be a modular Stanley sequence with character $\lambda$, shift index $\sigma$, and repeat factor $a_m$. Pick $k>0$ and $c$ such that $$\lambda \le c \le a_{m\cdot (2^k-1)}-\lambda.$$
Define
$$A_k^c:=\left\{a_i\mid 0\le i<m\cdot 2^k-\sigma\right\}
\cup\left\{c+a_i\mid m\cdot 2^k-\sigma\le i<m\cdot 2^{k+1}-\sigma\right\}.$$
Then, $A_k^c$ is 3-free and $S(A_k^c)$ is a pseudomodular Stanley sequence with core $S(A_k)$.
\end{theorem}

From Corollary \ref{cor:modgrowth}, it is simple to show that all pseudomodular sequences follow Type 1 growth. We revise Rolnick's conjecture as follows:

\begin{conjecture}
A Stanley sequence follows Type 1 growth if and only if it is pseudomodular.
\end{conjecture}


\section{Basic sequences}
\label{sec:basic}

We have noted that the sequence $S(0)$ consists of the sums of subsets of $(1,3,9,27,\ldots)$. We now introduce a class of modular Stanley sequences which generalize this behavior to sequences other than the powers of 3. 

\begin{definition}
We say that a Stanley sequence $S(A)=(a_n)$ is \emph{basic} if there exists a sequence $B=(b_k)$ and constant $\alpha\in \N$, such that (i) $b_k=\alpha\cdot 3^k$ for $k$ sufficiently large, and (ii) the elements $a_n$ correspond to sums of subsets of $B$, that is, we have
$$
S(A)=\left\{\sum \delta_k b_k\mid \delta_k\in \{0,1\}\text{ with $\delta_k\ne 0$ for finitely many $k$}\right\},
$$
where these sums are all distinct. In this case, we say that $B$ is the \emph{basis} of $S(A)$.
\end{definition}

In a sense, the basis $B$ of an basic Stanley sequence may be seen as a generalization of the powers of 3. Instead of writing the elements of the sequence in base $3$, we are able to write them simply in ``base $B$".  Just as $S(0)$ consists of those integers with only digits $0,1$ in base $3$, a basic sequence consists of those integers with only digits $0,1$ in base $B$.   Thus $S(0)$ is basic with basis $\{1,3,9,\ldots\}$.

\begin{example}
The sequence $S(0,1,7)$ is basic with basis $\{1,7,10,30,\ldots\}$. We may rewrite the first few terms of this sequence with reference to the basis, as follows:
\begin{align*}
S(0,1,7)&=0, 1, 7, 8, 10, 11, 17, 18, 30, 31, 37, 38,\ldots\\
&=0, 1, 7, 7+1, 10, 10+1, 10+7, 10+7+1, 30, 30+1, 30+7, 30+7+1,\ldots\\
\end{align*}
\end{example}

\begin{proposition}
Every basic Stanley sequence is independent (and thus modular).
\end{proposition}

\begin{proof}
Suppose that $S(A)$ is basic with basis $B$. For $k$ large enough, $b_k=\alpha\cdot 3^k$; therefore, for some $\kappa$ and all $k\ge \kappa$, we have $$b_k\ge \sum_{i=0}^{k-1} b_i.$$Since the sum of each subset of $\{b_0,b_1,\ldots,b_{k-1}\}$ is equal to a distinct element $a_n$ of $S(A)$, and there are $2^k$ such subsets, we conclude that $a_{2^k}=b_k=\alpha\cdot 3^k$ and $a_{2^k+i}=a_{2^k}+a_i$ for each $0\le i<2^k$. We conclude that $S(A)$ is independent, as desired.
\end{proof}

Using our formulation of modular Stanley sequences, we prove the following theorem.

\begin{theorem}
\label{thm:basic}
The sequence $B=(b_k)$ is a valid basis, provided that (i) $3^k$ is the largest power of $3$ dividing $b_k$ for each $k$, and (ii) $b_k=3^k$ for $k$ large enough. Note that $(b_k)$ need not be an increasing sequence for the first few values of $k$.
\end{theorem}

\begin{example}
The sequence $S(0, 9, 11, 12, 20)$ is basic with basis $\{11,12,9,27,\ldots\}$.
\begin{align*}
S(0, 9, 11, 12, 20) &= 0, 9, 11, 12, 20, 21, 23, 27, 32, 36,\ldots\\
&= 0, 9, 11, 12, 11+9, 12+9, 12+11, 27, 12+11+9, 27+9,\ldots
\end{align*}
\end{example}

\begin{proof}[Proof of Theorem \ref{thm:basic}.]
Suppose $B=(b_k)$ satisfies the given hypotheses. Let $N=b_m$ be large enough that $b_k=3^k$ for all $k\ge m$ and $\sum_{k=0}^{m-1}{b_k}<N$. Let $$A:=\left\{\sum_{k=0}^{m-1} \delta_k b_k\mid \delta_k\in \{0,1\}\right\}.$$Let $S$ be the sequence of the sums of subsets of $B$. We will prove that $S$ is a modular Stanley sequence modulo $N$, for which it suffices to prove that $A$ is a modular set modulo $N$.

For each integer $x\ge 0$, let $t_i(x)$ denote the $i$th digit of $x$ in base 3. For each $x\in A$, let $\delta_i(x)$ be the indicator for whether $b_i$ is in the subset of $B$ which is summed to obtain $x$.  Note that $t_i(x)$ is determined uniquely by $\delta_0(x),\delta_1(x),\ldots,\delta_i(x)$, since $x$ is a sum of $b_k$ and $t_i(b_k)=0$ for all $k>i$.

We first show that $A$ is 3-free modulo $N$. Suppose towards contradiction that $x,y,z\in A$ satisfy $2y-z\equiv x\pmod N$. Then, for each $i$, either $\{t_i(x),t_i(y),t_i(z)\}=\{0,1,2\}$ or $t_i(x)=t_i(y)=t_i(z)$. It is impossible to have $\{t_0(x),t_0(y),t_0(z)\}=\{0,1,2\}$, since the digit $t_0(\cdot)$ is determined solely by $\delta_0(\cdot)\in \{0,1\}$. Hence, $t_0(x)=t_0(y)=t_0(z)$ and so $\delta_0(x)=\delta_0(y)=\delta_0(z)$. Now, it is impossible to have $\{t_1(x),t_1(y),t_1(z)\}=\{0,1,2\}$, since the digit $t_1(\cdot)$ is determined solely by $\delta_0(\cdot),\delta_1(\cdot)\in \{0,1\}$ and we know $\delta_0(\cdot)$ is identical for $x,y,z$. Therefore, $t_1(x)=t_1(y)=t_1(z)$ and so $\delta_1(x)=\delta_1(y)=\delta_1(z)$. Continuing in this way, we conclude that all the digits of $x,y,z$ are identical, and so $x=y=z$, a contradiction.

Now suppose that $x\not\in A$ with $0\le x<N$. We must show that $x$ is covered by $A$ modulo $N$. We construct elements $y,z\in A$ stepwise such that $2y-z\equiv x$, in the following manner. At step $-1$, we start out with $y^{(-1)}=z^{(-1)}=0$ and $x^{(-1)}=0$. At step $j$, for $j\ge 0$, we create $y^{(j)},z^{(j)}\in A$ from $y^{(j-1)},z^{(j-1)}\in A$ so that $x^{(j)}:=2y^{(j)}-z^{(j)}$ agrees with $x$ in the digits $t_0(\cdot),\ldots,t_j(\cdot)$.

We assume recursively that $t_0(\cdot),\ldots,t_{j-1}(\cdot)$ have already been matched between $x^{(j-1)}$ and $x$ and that $\delta_j(y^{(j-1)}),\delta_j(z^{(j-1)})$ are both set to 0.  In order to construct $y^{(j)}$ and $z^{(j)}$, we set one, or both, or neither of $\delta_j(y^{(j)}),\delta_j(z^{(j)})$ to 1, while keeping all other $\delta_i(\cdot)$ fixed.

In order to determine how to set $\delta_j(y^{(j)}),\delta_j(z^{(j)})$, consider $t_j(x^{(j-1)})$. If $t_j(x^{(j-1)})=0$, then we set both to 0, since then $t_j(x^{(j)})=0$.  If $t_j(x^{(j-1)})=1$, then we can set either of $\delta_j(y^{(j)}),\delta_j(z^{(j)})$ to 1, while if $t_j(x^{(j-1)})=2$, then we set both of $\delta_j(y^{(j)}),\delta_j(z^{(j)})$ to 1; this again results in $t_j(x^{(j)})=0$.

The important part of this recursive procedure is that at step $j$, when we fix the digit $t_j(\cdot)$, we do not change any of the other digits $t_i(\cdot)$ for $i<j$.  This is because we are adding multiples of $b_j$, and $3^j \mid b_j$. We may affect some of the digits $t_i(\cdot)$, for $i>j$; however, we fix these digits later, up until step $N$, at which point $x^{(j)}$ matches $x$ in the first $N$ digits.  Hence, $x^{(j)}\equiv x\pmod N$. We have proven that $x$ is covered by $A$ modulo $N$, which completes our proof.
\end{proof}

Combining the arguments in Theorem \ref{thm:struc} and Theorem \ref{thm:basic} yields the following theorem.

\begin{theorem}
\label{thm:scalebasic}
Let $A$ be a modular set modulo $N$, and choose $\alpha\in\N$ such that $\gcd(\alpha,N)=1$.  Suppose that $B=(b_k)$ satisfies (i) $3^k$ is the largest power of $3$ dividing $b_k$ for each $k$, and (ii) $b_k=3^k$ for $k$ large enough. Let $S_B$ denote the Stanley sequence generated from the basis $B$. Then,
$$
\alpha\cdot A+N\cdot S_B
$$
is a modular sequence modulo $3^{\ell}\cdot N$ for all sufficiently large $\ell\in\N.$
\end{theorem}

Using the theory of basic sequences, we can prove a small result towards Rolnick's Conjecture 5.1 \cite{R}. Before recalling the conjecture, we first state a definition. Given a 3-free set $A$ with elements $a_0<\cdots<a_k$, define a \emph{completion} of $A$ to be a 3-free set $A'$ with elements $a_0<\cdots<a_k<\cdots<a_m$ such that $S(A')$ is regular. For instance, $\{0,4,7\}$ and $\{0,4,9\}$ are two completions of $\{0,4\}$.

\begin{conjecture}[Conjecture 5.1 in \cite{R}]
Every 3-free set has a completion.
\end{conjecture}

\begin{proposition}
Let $a_1<\cdots<a_n\in\N$ such that $\nu_3(a_i)\ne\nu_3(a_j)$ for $i\ne j$ where $3^{\nu_3(\cdot)}$ is the highest power of $3$ dividing a number. Also suppose that $a_1+a_2>a_n$. Then the set $\{0,a_1,\ldots,a_n\}$ has a completion.
\end{proposition}
\begin{proof}
Let $c\in\N$ such that $c>a_n$ and $\nu_3(c)=0$. Then Theorem \ref{thm:basic} implies that the following set is a basis for an independent Stanley sequence beginning with the terms $0,a_1,\ldots,a_n:$
$$\{c\cdot 3^j\mid 3^j<a_n,\ j\ne\nu_3(a_i)\text{ for }1\le i\le n \}\cup\left\{a_1,\ldots,a_n \right\}\cup\{3^j\colon 3^j>a_n \}.$$
The condition $a_1+a_2>a_n$ is necessary to ensure that the first $n+1$ terms of the basic sequence generated by the above basis are equal to $0,a_1,\ldots,a_n$.
\end{proof}

This proposition clearly implies the following corollary.

\begin{corollary}
\label{cor:completion}
Let $a<b\in\N$ with $\nu_3(a)\ne\nu_3(b)$, then the set $\{0,a,b\}$ has a completion.
\end{corollary}


\section{Stanley sequences with large gaps between consecutive terms}
\label{sec:gaps}

For arbitrary Stanley sequences $S(A)=(a_n)$, the gaps $a_{n+1}-a_n$ between consecutive elements grow on average, but small gaps do still occur even for large $n$.  In \cite[Problem 4, p. 126]{ELRSS}, the authors ask whether there exists a Stanley sequence $(a_n)$ such that $\liminf_{n\rightarrow\infty}\left(a_{n+1}-a_n\right)=\infty$. This is easily seen to be false for modular sequences, and we believe it to be false for general Stanley sequences.  We here prove the weaker statement that there exist Stanley sequences $(a_n)$ such that $\liminf_{n\rightarrow\infty}\left(a_{n+1}-a_n\right)$ is arbitrarily large.

\begin{proposition}
\label{prop:gaps}
For $m$ a nonnegative integer, define
$$
A_m:=\left\{\sum_{b\in B} b\mid B\subseteq \bigcup_{i=0}^m \left\{2(2^i\cdot 29^{m-i}),6(2^i\cdot 29^{m-i}),11(2^i\cdot 29^{m-i})\right\}\right\}.
$$
Then $A_m$ is a modular set modulo $29^{m+1}$.
\end{proposition}

\begin{proof}
We proceed by induction on $m$.

\medskip

{\bf Base case:} $m=0$

A routine computation shows that $A_0=\{0,2,6,8,11,13,17,19\}$ is 3-free modulo $29$ and covers all elements $x\in\{0,\ldots,28\}\backslash A_0$ modulo $29$. Thus, $A_0$ is a modular set modulo $29$.

\medskip

{\bf Induction step:} Suppose $A_m$ is a modular set modulo $29^{m+1}$ and consider $A_{m+1}$. First we want to show that $A_{m+1}\subset\{0,\ldots,29^{m+2}-1\}$. We can easily see that 
$$
A_{m+1}=2^{m+1}\cdot A_0+29\cdot A_m.
$$

Hence,
$$
\max\left(A_{m+1}\right)=\max\left(2^{m+1}\cdot A_0+29\cdot A_{m} \right)=2^{m+1}\cdot 19+29\cdot\max(A_m).
$$
By recursion, this expression equals
\begin{align}\label{eq:geosum}
&\sum_{i=0}^{m+1}{19\cdot 2^{i}\cdot 29^{m+1-i}}=29^{m+1}\cdot 19\sum_{i=0}^{m+1}\left(\frac{2}{29}\right)^i\nonumber \\ &<29^{m+1}\cdot 19\sum_{i=0}^\infty\left(\frac{2}{29}\right)^i=29^{m+1}\cdot 19\cdot \frac{29}{27}<29^{m+2}.
\end{align}
Therefore, $A_{m+1}\subset\{0,\ldots,29^{m+2}-1\}$.

Now we will show that $A_{m+1}$ is 3-free modulo $29^{m+2}$. If not, there exists a 3-AP $z<y<x\in A_{m+1}$ modulo $29^{m+2}$. Write $x=2^{m+1}x_0+29x_m$, for $x_0\in A_0$ and $x_m\in A_m$, and define $y_0,z_0,y_m,z_m$ similarly. Since $z,y,x$ form a 3-AP modulo $29^{m+2}$, we know that $z_0,y_0,x_0$ form a 3-AP modulo $29$. Hence, $x_0=y_0=z_0$ since $A_0$ is 3-free modulo $29$. We conclude that $z_m,y_m,x_m$ is a 3-AP modulo $29^{m+1}$. Since $A_m$ is 3-free modulo $29^{m+1}$, this means that $x_m=y_m=z_m$, and thus that $A_{m+1}$ is 3-free modulo $29^{m+2}$.

Now, we must show that every element $x\in\{0,\ldots,29^{m+2}-1\}\backslash A_{m+1}$ is covered by $A_{m+1}$ modulo $29^{m+2}$. Let $x_0\in \{0,1,\ldots,28\}$ be the unique value such that $2^{m+1}x_0\equiv x\pmod{29}$. If $x_0\in A_0$, set $y_0=z_0=x_0$. If not, pick $z_0,y_0\in A_0$ that cover $x_0$ modulo $29$. Now, define $x_m$ by $x=2^{m+1}(2y_0-z_0)+29x_m$. If $x_m\in A_m$, then set $y_m=z_m=x_m$. Else, we know by our inductive hypothesis that we can pick $z_m,y_m\in A_m$ that cover $x_m$ modulo $29^{m+1}$. Setting $y=2^{m+1}y_0+29y_m$ and $z=2^{m+1}z_0+29z_m$, we observe that $z,y\in A_{m+1}$ cover $x$ modulo $29^{m+2}$, completing our induction.
\end{proof}

Given a Stanley sequence $S(A)=(a_n)$, we write $\gap(A)$ for $\liminf_{n\rightarrow\infty}{\left(a_{n+1}-a_n\right)}$.

\begin{corollary}
\label{cor:liminf}
The modular Stanley sequence $S(A_m)=(a_n)$ has $\gap(A)=2^{m+1}$.
\end{corollary}

\begin{proof} We again proceed by induction on $m$.

{\bf Base case:} For $m=0$, it is readily verified that we have 
$$
\gap(0,2,6,8,11,13,17,19)=2.
$$ 

\medskip

{\bf Induction step:} Suppose that we have $\gap(A_m)=2^{m+1}$. In order to prove that $\gap(A_{m+1})=2^{m+2}$, we need only show that every two consecutive terms of $A_{m+1}\cup\{29^{m+2}\}$ are separated by at least $2^{m+2}$.

Observe that $A_{m+1}=2\cdot A_{m}+\{0,2\cdot 29^{m+1}\}+\{0,6\cdot 29^{m+1}\}+\{0,11\cdot 29^{m+1}\}$. By our inductive hypothesis, we have $\liminf{\left(2\cdot A_m\right)}=2^{m+2}$. Therefore, the gaps between consecutive elements of $A_{m+1}\cup\{29^{m+2}\}$ are at least $2^{m+2}$ if the following conditions hold:

\begin{align*}
2\cdot 29^{m+1}-\max{\left(2\cdot A_m\right)}&\geq 2^{m+2},\\
6\cdot 29^{m+1}-\max{\left(2\cdot A_m+ \{0,2\cdot 29^{m+2}\} \right)}& \geq 2^{m+2},\\
11\cdot 29^{m+1}-\max{\left(2\cdot A_m+ \{0,2\cdot 29^{m+1}\}+\{0,6\cdot 29^{m+1}\} \right)}& \geq 2^{m+2},\\
29^{m+2}-\max{\left(A_{m+1}\right)}&>2^{m+2}.
\end{align*}
For the first condition, we use a recursive calculation similar to \eqref{eq:geosum} to compute:
$$
29^{m+1}-\max\left(A_m\right)\geq 29^{m+1}- \frac{19}{27}\cdot 29^{m+1}>2^{m+1},
$$
and thus $2\cdot 29^{m+1}-\max{\left(2\cdot A_m\right)}\geq 2^{m+2}$. The other conditions follow similarly. We conclude that $\liminf{S(A_{m+1})}=2^{m+2}$, completing the induction.
\end{proof}


\section{Generalization to $p$-free sequences}
\label{sec:pfree}

One can easily generalize the greedy algorithm studied by Odlyzko and Stanley to produce $p$-free sequences where $p>2$ is a prime.

\begin{definition}
Let $A=\{a_0,\ldots,a_k\}$ be a $p$-free set of nonnegative integers satisfying $a_0<\cdots<a_k$. We define the \emph{$p$-Stanley sequence} $S(A)=(a_n)$ generated by $A$ recursively as follows. If $a_0<\cdots<a_n$ have already been defined, then $a_{n+1}$ is the smallest positive integer greater than $a_n$ such that $\{a_0,\ldots,a_n,a_{n+1}\}$ is $p$-free.
\end{definition}

In this section, we will generalize the notion of modular Stanley sequences to modular $p$-Stanley sequences. Throughout the rest of this section $p$ will denote an odd prime. Though one can certainly consider $p$-modular sets and sequences for $p$ \emph{not} prime, these sequences will not have the regular structure given in Theorem \ref{thm:struc}.

\begin{definition}
Let $A$ be a set of integers and $x$ be an integer.  We say that $x$ is \emph{$p$-covered by $A$} if there exist $x_1<\cdots<x_{p-1}$ in $A$ such that $x_1,\ldots,x_{p-1},x$ form a $p$-term arithmetic progression ($p$-AP).

Suppose that $N$ is a positive integer with $A\subseteq \{0,\ldots,N-1\}$. Then, we say that an element $x$ is \emph{$p$-covered by $A$ modulo $N$} if there exists $x_1,\ldots,x_{p-1}\in A$ with $x_1<\cdots<x_{p-1}$ such that $x_1,\ldots, x_{p-1},x$ form a $p$-AP modulo $N$.
\end{definition}

\begin{definition}
Fix a positive integer $N\ge 1$. Suppose there exists a set $A\subset\{0,\ldots,N-1\}$ containing $0$ such that $A$ is $p$-free modulo $N$ and all $x\in\{0,\ldots,N-1\}\backslash A$ are $p$-covered by $A$ modulo $N$. Then $A$ is said to be \emph{modular $p$-free set} modulo $N$, and $S_p(A)$ is said to be a \emph{modular $p$-Stanley sequence} modulo $N$.
\end{definition}

Before we describe how the theory of modular Stanley sequences generalizes to the theory of modular $p$-Stanley sequences, let us prove a well-known lemma. The existence of these sequences provided motivation for Szekeres in his conjecture about the asymptotic behavior of $r_p(n)$ in \cite{ET}.
\begin{lemma}
Let $p>2$ be a prime, then $S_p(0)$ consists exactly of the integers $x\ge 0$ such that $x$ contains only the ``digits'' $\{0,\ldots,p-2\}$ in its base $p$ expansion.
\end{lemma}

\begin{proof}
Let $t_i^p(x)$ denote the $i$th digit in the base $p$ expansion of $x$. Let 

$$
S_p=\left\{x\ge 0\ :\ t_i^p(x)\in\{0,\ldots,p-2\}\ \forall\  i\right\}.
$$

First we will show that $S_p$ is $p$-free. Suppose towards contradiction that $x_0,\ldots,x_{p-1}\in S_p$ form a $p$-AP, where $x_{p-1}$ is minimal. If $t_0^p(x_0)=\ldots =t_0^p(x_{p-1})$, then $\frac{x_0-t_0^p(x_0)}{p},\ldots,\frac{x_{p-1}-t_0^p(x_{p-1})}{p}$ is a $p$-AP in $S_p$ with strictly smaller $p^{th}$ term, a contradiction. Hence, the $t_0^p(x_j)$ are not identical. Therefore, the $t_0^p(x_j)$ must attain every value in $\{0,\ldots,p-1\}$, which contradicts the fact $t_i^p(x_j)\in\{0,\ldots,p-2\}$ for all $i$ and all $j$. Hence $S_p$ is $p$-free.

Now, we will show that all elements $x\in\N\backslash S_p$ are $p$-covered by $S_p$. Let $x\in\N\backslash S_p$. We choose integers $x_0,\ldots,x_{p-2}$ as follows. If $t_i^p(x)\ne p-1$, then set $t_i^p(x_j)=t_i^p(x)$ for all $0\le j\le p-2$. If $t_i^p(x)=p-1$, then set $t_i^p(x_j)=j$ for all $0\le j\le p-2$. Using this construction it is easy to see that $x_0<\cdots<x_{p-2}<x$ is a $p$-AP.  Thus, $S_p(0)=S_p$.
\end{proof}

It is simple to verify that $S_p(0)$ is a modular $p$-Stanley sequence modulo $1$. This leads us to the following structure theorem for modular $p$-Stanley sequences.

\begin{theorem}\label{thm:pstruc}
Fix a positive integer $N\ge 1$ and suppose $A\subset\{0,\ldots,N-1\}$ contains $0$, is $p$-free modulo $N$, and all $x\in\{0,\ldots,N-1\}\backslash A$ are $p$-covered by $A$ modulo $N$. Then
$$
S_p(A)=A+N\cdot S_p(0).
$$
More generally, if $\alpha\in \N$ is such that $\gcd(\alpha,N)=1$, then
$$
\alpha\cdot A + N\cdot S(0)
$$
is a modular Stanley sequence modulo $3^{\ell}\cdot N$ for some integer $\ell\ge 0$.
\end{theorem}
\begin{proof}
The proof closely follows that of Theorem \ref{thm:struc}.
\end{proof}

\begin{example}
The 5-Stanley sequence $S_5(0,3)$ is a modular sequence modulo 25. Note, for instance, that the terms $a_{16}$ through $a_{31}$ equal the terms $a_0$ through $a_{15}$, translated by $a_{16}=25$.

\begin{align*}
S_5(0,3)&=0, 3, 4, 5, 6, 8, 9, 10, 11, 13, 14, 15, 16, 18, 19, 21,\\
& 25, 28, 29, 30, 31, 33, 34, 35, 36, 38, 39, 40, 41, 43, 44, 46,\ldots
\end{align*}
\end{example}

\begin{problem}
Classify all $m$, as a function of $p$, such that $S_p(0,m)$ is ``well-behaved'' (where \emph{modular} represents a good definition of well-behaved). For $p=3$, Odlyzko and Stanley conjectured that the only such $m$ are $3^n$ and $2\cdot 3^n$. For $p=5$, our code suggests that the possible values $m\le 100$ are as follows (written in base 5):
\begin{align*}
&1,3,4,10,22,23,24,30,32,33,34,40,42,43,44,\\
&100,122,124,130,132,133,134,140,142,\\
&212,213,214,220,222,223,224,230,232,233,234,240,242,243,244,\\
&300,312,313,314,320,322,323,324,330,332,333,334,340,342,343,344.
\end{align*}
\end{problem}

We now generalize the concept of a basic Stanley sequence to $p$-Stanley sequences.

\begin{definition}
We say that an $p$-Stanley sequence $S_p(A)=(a_n)$ is \emph{basic} if there exists a sequence $B=(b_k)$ and constant $\alpha\in \N$ such that (i) $b_k=\alpha\cdot p^k$ for $k$ sufficiently large, and (ii) the following property is satisfied:
$$
S_p(A)=\left\{\sum \delta_k b_k\mid \delta_k\in \{0,1,\ldots,p-2\}\text{ with $\delta_k\ne 0$ for finitely many $k$}\right\},
$$
where these sums are all distinct. In this case, we say that $B$ is the \emph{basis} of $S_p(A)$.
\end{definition}

Observe that $S_p(0)$ is basic with basis $\{1,p,p^2,\ldots\}$. We can create many more basic $p$-Stanley sequences with the following theorem, which is proven in the manner of Theorem \ref{thm:basic}.

\begin{theorem}\label{thm:pbasic}
The sequence $B=(b_k)$ is a valid basis, provided that (i) $p^k$ is the largest power of $p$ dividing $b_k$ for each $k$, and (ii) $b_k=p^k$ for $k$ large enough. Note that $(b_k)$ need not be an increasing sequence for the first few values of $k$.
\end{theorem}

\begin{example}
The sequence $(7,5,25,\ldots)$ is a basis for the $5$-Stanley sequence that starts out:
\begin{align*}
&(0, 5, 7, 10, 12, 14, 15, 17, 19, 21, 22, 24, 25, \ldots)\\
&(0, 5, 7, 2\cdot 5, 7+5, 2\cdot 7, 3\cdot 5, 7+2\cdot 5, 2\cdot 7+5, 3\cdot 7, 7+3\cdot 5, 2\cdot 7+2\cdot 5, 25,\ldots).
\end{align*}
\end{example}

Theorem \ref{thm:scalebasic} also generalizes naturally to the $p$-free setting.


\section{Future directions}
\label{sec:future}

The study of Stanley sequences has been marked by breakthrough results in which new and unexpected forms of structure are shown to be possible.  While most Stanley sequences are disorderly and their behavior is a mystery, the simple sequence $S(0)$ is the starting point that shows that some beautiful structures are possible. Odlyzko and Stanley discovered orderly sequences of the form $S(0,3^n)$ and $S(0,2\cdot 3^n)$, which Rolnick generalized in turn with the discovery of independent and regular sequences. In this paper, we have extended this definition in turn to the modular and pseudomodular sequences.

It is difficult to assess what novel structures may yet be discovered in Stanley sequences. While we believe that modular and pseudomodular sequences comprise the totality of sequences with Type 1 growth, it is possible that some wholly new class of well-structured Stanley sequences may be derived using fresh perspectives or greater computational power. It is, for example, possible that there exists a Stanley sequence that is in some sense ``denser'' than the sequence $S(0)$. We make this notion explicit in the following conjecture.

\begin{conjecture}
\label{conj:denser}
Let $S(0)=(s_n)$. Then, there exists a modular sequence $S(A)=(a_n)$, with $a_n<s_n$ for all sufficiently large $n$.
\end{conjecture}

Consider the modular sequence $S(A)=(a_n)$ with
$$A=\{0, 5, 7, 8, 12, 13, 15, 20, 29, 36, 44, 55, 62, 63, 70\},$$
which we presented in Table \ref{table:notindep}. This sequence is modular with modulus 79, and it satisfies $a_n<s_n$ for infinitely many $n$. (No examples of such a sequence were previously known, and we know of no independent Stanley sequence which satisfies this property.) It appears to be considerably more difficult, however, to identify a modular sequence for which $a_n<s_n$ is always true for $n$ large enough.

While independent Stanley sequences are modular, the cardinality of their modular set is always a power of 2. A key insight of this paper is that modular sets can have other cardinalities; this opens the way to general modular sequences. We now conjecture that, in fact, every sufficiently large integer is the cardinality of some modular set.

\begin{conjecture}
\label{conj:allsizes}
There exists an $n_0\in\N$ such that for all $n\ge n_0$, there exists an integer $N_n\in\N$ and a modular set $A_n$ modulo $N_n$ such that $|A_n|=n$.
\end{conjecture}

Clearly there exist modular sets of cardinality $n$ for all $n=2^m$ where $m\in\N$. More generally, if there exists a modular set $A$ of cardinality $n$, then we can use prefix subsequences of $S(A)$ to create modular sets of cardinality $2^m\cdot n$ for all $m\in \N$. Also, if we have two modular sets $A_1,A_2$ of cardinality $n_1,n_2$ respectively, then their product $S_1\otimes S_2$ is a modular set of cardinality $n_1\cdot n_2$.

Finally, we expect that many properties of modular $3$-Stanley sequences will generalize to modular $p$-Stanley sequences. However, constructing $p$-Stanley sequences requires far more computational power; the natural algorithm for constructing the first $n$ terms takes time $O(n^{p-1})$. Therefore, it becomes increasingly difficult to construct interesting examples of $p$-Stanley sequences. We offer the following conjectures.

\begin{conjecture}
For any $n\in\N\cup\{0\}$, there exists a modular $p$-free set $A$ such that 
$$|A|\ne (p-1)^n.
$$
\end{conjecture}

The only resolved case of this conjecture is when $p=3$, and examples of such modular sets can be found in Table \ref{table:notindep}. The greedy algorithm easily produces modular $p$-free sets with cardinality $(p-1)^n$, but producing other modular $p$-free sets is harder. The only known method to produce such a set is by exhaustive search. Of course, once one finds one such set, producing others is easy. Just as in the $3$-free case, we expect all sufficiently large integers to appear as the cardinality of a modular $p$-free set. We state here the $p$-free generalizations of Conjectures \ref{conj:denser} and \ref{conj:allsizes}.

\begin{conjecture}
Let $S_p(0)=(s_n)$. Then, there exists a modular $p$-Stanley sequence $S_p(A)=(a_n)$ with $a_n<s_n$ for all sufficiently large $n$.
\end{conjecture}

\begin{conjecture}
For every odd prime $p$ there exists a natural number $n_p$ such that for all $n\ge n_p$, there exists a modular $p$-free set $A$ with $|A|=n$.
\end{conjecture}

\bibliography{stanley}
\bibliographystyle{plain}

\end{document}